\documentclass[reqno,twoside,a4paper,11pt]{amsart}
\usepackage[margin=3cm,headheight=14pt,centering=true,heightrounded=true]{geometry}
\usepackage{upref,amsmath,amssymb,amsthm,amsfonts}
\usepackage{stmaryrd} 
\usepackage[utf8]{inputenc}
\usepackage{fancyhdr}
\usepackage{enumerate}
\usepackage{tikz-cd} 

\usepackage[colorlinks=true,linkcolor=blue,citecolor=blue]{hyperref}

\usepackage{math699}

\usepackage[scaled]{helvet} 
\usepackage{courier} 
\usepackage[mathbf,mathcal]{euler}
\DeclareSymbolFont{rsfs}{U}{rsfs}{m}{n}
\DeclareSymbolFontAlphabet{\mathrsfs}{rsfs}

\numberwithin{equation}{section}

\newcommand{\simpleenum}[1]{\par\indent{\upshape\bfseries #1}}

\theoremstyle{plain}
\newtheorem{theorem}{Theorem}
\newtheorem{prop}[theorem]{Proposition}
\newtheorem{lemma}[theorem]{Lemma}

\theoremstyle{remark}
\newtheorem*{remark}{Remark}

\def\mlSetRCSInfo#1{\def\mlRCSInfo{#1}}
\mlSetRCSInfo{13.03.2024 13:58}
}

\newcommand{\schatten}[2]{\mathrsfs{B}^{#1}(#2)}
\newcommand{\BH}[1]{\schatten{#1}{\mathcal{H}} }
\newcommand{\LH}{\BH{}}

\newcommand{\compact}{\mathrsfs{K}}
\newcommand{\bone}{\BH 1}

\newcommand{\detm}[1][m]{\det\nolimits_{#1}}
\newcommand{\logm}[1][m]{\log\nolimits_{#1}}
\newcommand{\detF}{\det\nolimits_F}
\newcommand{\yaj}{y_{\xi,j}}
\newcommand{\yajk}{\yaj^k}
\newcommand\Trb[1]{\Tr\bigl(#1\bigr)}

\newcommand\britz{\cite{BCG21}}

\newcommand{\setr}{\{1,\ldots,r\}}
\newcommand{\setj}{\{1,\ldots,j\}}

\newcommand{\avec}{a_1,\ldots,a_r}

\newcommand{\tvec}{t_1,\ldots,t_r}
\newcommand{\Avec}{A_1,\ldots,A_r}

\newcommand{\mfMjr}{\mfM(j,r)}

\newcommand{\Acomm}{[\cA,\cA]}

\newcommand{\zen}[1]{\cZ#1} 
\newcommand{\Qcs}[1]{\cQ #1 } 
\newcommand{\Qcsp}[1]{\cQ(#1)} 

\newcommand{\QA}{\cQ\cA}
\newcommand{\CA}{\cC\cA}
\newcommand{\ZA}{\cZ\cA}

\newcommand{\PFree}[2]{#1\langle #2 \rangle}
\newcommand{\Qab}{\PFree{\Q}{a,b}}

\newcommand{\Qar}{\PFree{\Q}{a_1,\ldots,a_r}}

\newcommand{\KFreeExpanded}{\PFree{\K}{\avec}}
\newcommand{\KFree}{\cA}

\newcommand{\FPower}[2]{ {#1} \llbracket #2 \rrbracket }

\newcommand{\KFormal}[1]{\KFree\llbracket #1 \rrbracket }

\newcommand{\Att}{ \KFormal t }
\newcommand{\Attn}[2][t]{ \cA_{#2}\llbracket #1 \rrbracket }
\newcommand{\QAtt}[1][t]{ \cQ\cA\llbracket #1 \rrbracket }
\newcommand{\QAttn}[2][t]{ \cQ\cA_{#2}\llbracket #1 \rrbracket }
\newcommand{\ZAtt}{ \cZ\cA\llbracket t \rrbracket }
\newcommand{\ZAttn}[2][t]{ \cZ\cA_{#2}\llbracket #1 \rrbracket }
\newcommand{\CAtt}{ \cC\cA\llbracket t \rrbracket }

\newcommand\trc[1]{\tau_{#1}}
\newcommand\trA{\tau_{\cA}}
\newcommand\trAt[1][t]{\sigma_{\KFormal #1}}
\newcommand{\trC}{\trc{\cA}}
\newcommand{\trCt}{\trc{\cA[t]}}

\newcommand{\Xmr}[1]{X_{m,r}(#1)}
\newcommand{\Xmrr}{X_{m,r}}
\newcommand{\Zmr}[1]{Z_{m,r}(#1)}
\newcommand{\Zmrr}{Z_{m,r}}
\newcommand{\Xtmr}[1]{\widetilde X_{m,r}(#1)}
\newcommand{\Xtmrr}{\widetilde X_{m,r}}

\newcommand{\msum}[3][1]{\sum\limits_{#2=#1}^{#3} }
\newcommand{\mpi}[3][1]{\prod\limits_{#2=#1}^{#3}}

\newcommand{\lncf}[1][j]{\frac{(-1)^{#1-1}}{#1} } 
\newcommand{\lncfpl}[1][j]{\frac{(-1)^j}{j} }

\newcommand{\mfP}{\mathfrak P}
\newcommand{\mfM}{\mathfrak M}

\newcommand{\Nrs}{\N^{r*}}

\begin{document}
\title[Product formula for regularized Fredholm determinants]{%
    The product formula for regularized Fredholm determinants: two new proofs}

\author{Nikolaos Koutsonikos-Kouloumpis}
\address{
Department of Mathematics,
University of Patras,
26504 Patras,
Greece}
\email{up1019669@ac.upatras.gr}

\author{Matthias Lesch}
\address{Mathematisches Institut,
Universit\"at Bonn,
Endenicher Allee 60,
53115 Bonn,
Germany}
\email{lesch@math.uni-bonn.de}
\urladdr{www.math.uni-bonn.de/people/lesch}

\thanks{This paper is based on the first named author's Master's thesis \cite{KK21}
written under the supervision of the second named author.
Both authors gratefully acknowledge the support of the Bonn International
Graduate School (BIGS)}

\makeatletter
\@namedef{subjclassname@2020}{%
    \textup{2020} Mathematics Subject Classification}
\makeatother

\begin{abstract} 
For an $m$-summable operator $A$ in a separable Hilbert space the higher
regularized Fredholm determinant $\detm(I+A)$ generalizes the classical
Fredholm determinant. Recently, Britz et al presented a proof of a product
formula
\[
  \detm\bl (I+A)\cdot(I+B) \br
    = \detm (I+A) \cdot \detm (I+B) \cdot \exp\Trb{X_m(A,B)},
\]
where $X_m(A,B)$ is an explicit polynomial in $A,B$ with values in the
trace class operators. If $m=1$ then $X_1(A,B)=0$, hence the formula
generalizes the classical determinant product formula.

One of the purposes of this note is to present two very simple alternative proofs
of the formula. The first proof is a priori analytic and makes use of 
the fact that $z\mapsto \detm(I+zA)$ is holomorphic, while the second proof
is completely algebraic. The algebraic proof has, in our opinion, 
some interesting aspects in its own about the trace and commutators.

Secondly, we extend the above mentioned formula to several factors
\[
  \detm\Bl \prod_{l=1}^r (I+A_l) \Br
    = \left( \prod_{l=1}^r \detm (I+A_l) \right)
       \cdot \exp\Trb{\Xmr\Avec}.
\]
The latter is more than just a straightforward generalization
as we will gain more insights into the combinatorics behind it. Also
we will present an algebraized version of the analytic proof in the
language of formal power series. 
The upshot is that the two identities are just combinatorial in nature.
\end{abstract}

\maketitle
\tableofcontents


%

	
\section{Introduction}
\label{s:intro}
%
\subsection{Notations and conventions}\label{ss:Not}
We denote by $\N=\{0,1,2,\ldots\}, \Z, \R, \C$ the natural numbers,
integers, real and complex numbers resp.
The cardinality of a set $\xi$ is denoted by $|\xi|$.

In Part \ref{part2} we will also make use of multiindex notation.
For a multiindex $\ga=(\ga_1,\ldots,\ga_r)\in\N^r$ we write $|\ga| = 
\ga_1+\ldots+\ga_r$.

When dealing with noncommuting variables, product notation means 
that the product is taken in the given order,
i.e. $\prod_{l=1}^N x_j = x_1\cldots x_N$ in this order. 

Concerning Hilbert spaces and trace ideals our standard references are
\cite{Sim05} and \cite{Ped89}. Hilbert spaces will be denoted by
calligraphic letters, e.~g.  $\cH$. $\LH, \compact(\cH)$ denote the algebra
of bounded resp. compact linear operators on $\cH$. $\Tr=\Tr_{\cH}$ denotes
the trace on nonnegative operators in $\LH$ resp. on all operators in the
trace ideal $\bone$. For a compact operator $K\in\compact(\cH)$ we adopt
the convention of Simon's book \cite{Sim05} and denote by 
$\gl_n(K), 1\le n\le N\in\N\cup\{\infty\}$, the $n$-th nonzero eigenvalue
(in some enumeration with $|\gl_n|$ in decreasing order, counted with
multiplicity) of $K$. $N$ is a natural number or $\infty$.

By $\BH p$ we denote the von Neumann--Schatten ideal of $p$--summable
operators. Recall that $T\in\LH$ is $p$--summable if 
\[
    \Trb{ |T|^p } = \Trb{ (T^*T)^{p/2} } =  
         \sum_{n=1}^\infty \gl_n(T^*T)^{p/2} <\infty.
\]         

%
\subsection{Main result}
\label{ss:MR}
This paper is inspired by the recent work \britz\ of Britz et al. on the
product formula for regularized Fredholm determinants. Namely, the
classical theory of trace ideals (see \cite{Sim05,Sim77} but see also the
more historical quotes after (1.3) in \britz) allows to generalize the
notion of Fredholm determinant to operators which differ from the identity
by a $p$-summable operator. More concretely, let first $A$ be a trace class
operator in a separable Hilbert space $\cH$. Then
\begin{equation}\label{eq:detF}
	\detF (I + A)  = \sum_{k=0}^{\infty} 
         \Trb{\Lambda^k A}=\prod_{n=1}^N \bl 1+ \gl_n(A) \br
\end{equation}
is called the \emph{Fredholm determinant} of $I+A$, sometimes the latter
being called of \emph{determinant class}. The Fredholm determinant retains
two important properties of the ordinary Linear Algebra determinant.
Firstly, there is a product formula
\begin{equation}\label{eq:product}
   \detF\bl (I+A)\cdot (I+B) \br = \detF(I+A) \cdot \detF(I+B),
\end{equation}
to which we come back in due course and secondly the function $z\mapsto
\detF( I+ z\cdot A)$ is an entire function\footnote{Of course, if
$N<\infty$ it is even a polynomial.} of genus $0$ \cite[Chap. 5, Sec.
2.3]{Ahl:CA} with zeros exactly in $-1/\gl_n(A)$, where $\gl_n(A),1\le n
\le N$, are the nonzero eigenvalues of $A$ counted with multiplicity as
explained above in Section \ref{ss:Not}.  The second property supports
the intuition that $z\mapsto \det(I+z\cdot A)$ plays the role of the
``characteristic polynomial'' of $A$. 

If $A\in\BH m$ is only $m$--summable, $m\in\N, m\geq 2$, then the product
\eqref{eq:detF} will not converge. One therefore employs the classical
Weierstraß method of convergence generating factors. Namely, it turns out
that the operator
\begin{equation}\label{eq:I3}
    (I+A) \cdot \exp\Bl \sum_{j=1}^{m-1}\frac{(-A)^j}{j} \Br
\end{equation}
is of determinant class, i.~e. differs from the identity by a trace class
operator and hence one obtains a regularized version of the Fredholm
determinant by putting
\begin{equation}
\begin{split}
\detm(I+A)
    &:= \detF\Bl (I+A) 
             \cdot \exp\Bl\sum_{j=1}^{m-1}\frac{(-A)^j}{j} \Br \Br\\
    & = \prod_{n=1}^N \Bl (1+ \gl_n) 
             \cdot \exp\Bl\sum_{j=1}^{m-1} \frac{(-\gl_n)^j}{j} \Br\Br,
             \quad \gl_n = \gl_n(A).
\end{split}\label{eq:I4}
\end{equation}
Note that the function of $z$
\begin{equation}\label{eq:canonical}
    \detm(I+z\cdot A) = \mpi n N \biggl[ (1+ z\cdot \gl_n) \cdot
       \exp\Bl\sum_{j=1}^{m-1} \frac{(-\gl_n z)^j}{j} \Br \biggr]
\end{equation}    
is a canonical Weierstraß product of genus $m-1$ \cite[Chap. 5, Sec.
2.3]{Ahl:CA} with zeros exactly in $-1/\gl_n$, hence it is an entire
function of order at most $m$, cf. also \cite[Sec. 4]{HarLes22}.
Of this we will only use that the function
is holomorphic in a neighborhood of $z=0$.

The product formula, however, is more involved.
\begin{theorem}[Product Formula {\cite[Theorem 1.1]{BCG21}}]\label{thm:main}
There exist explicit (see below) polynomials $X_m(a,b)$ over $\Q$ in the
free polynomial algebra $\Qab$ in two noncommuting indeterminates $a,b$
such that the following holds:

Let $\cH$ be a separable Hilbert space and let 
$A, B\in \BH m, m\in\N, m\ge 1$, be $m$-summable operators. Then $X_m(A,B)$
is of trace class and one has
\begin{equation}\label{eq:main}
  \detm\bl (I+A)\cdot(I+B) \br
    = \detm (I+A) \cdot \detm (I+B) \cdot \exp\Trb{ X_m(A,B) }.
\end{equation}
\end{theorem}
\begin{remark}
\simpleenum{1. } Theorem \ref{thm:main} is essentially a reformulation
of \cite[Theorem 1.1]{BCG21}. But there are predecessors. As pointed out
to us by Rupert Frank, Hansmann \cite[Lemma 1.5.10]{Han10} considered the
case where $A$ (or $B$) is of finite rank.

\simpleenum{2. } Of course for an element $f(a,b)\in\Qab$ we denote by
$f(A,B)$ the operator obtained by inserting $A$ for $a$ and $B$ for $b$.
\end{remark}

To describe the polynomials $X_m(a,b)$ we need to introduce one more
notation from \britz. For a subset $\xi\subset\setj$ one puts
\begin{equation}\label{eq:I6}
    \yajk(a,b):=\begin{cases} ab,  & k\in \xi, \\
                                a+b, & k\not\in \xi,
           \end{cases},\quad 
             \yaj(a,b):=\prod_{k=1}^j\yajk(a,b).
\end{equation}
Then
\begin{equation} \label{eq:I7}
    (a+b+ab)^j=\sum_{\xi\subset\{1,..,j\}}\yaj(a,b).
\end{equation}
\begin{prop}\label{p:2} The polynomial $X_m(a,b)$ in Theorem \ref{thm:main}
is explicitly given by
\begin{equation}\label{eq:I8}
    X_m(a,b) = \sum_{j=1}^{m-1} \frac{(-1)^j}{j}
       \sum_{\substack{
               \xi\subset\setj\\
                j+|\xi|\ge m}
            }
        \yaj(a,b).
\end{equation}
\end{prop}
\begin{remark} Note that $\yaj(a,b)$ is a product of $|\xi|+j$ factors.
Hence $\yaj(A,B)\in\bone$ if $A,B\in\BH m$ and $|\xi|+j\ge m$.
We will from now on reserve the notation $X_m(a,b)$ for the right hand
side of \eqref{eq:I8}. It is then to be proved that \eqref{eq:main}
holds.

Furthermore, $X_m(a,b)$ has degree $2m-2$ and furthermore each monomial in
$X_m(a,b)$ has degree at least $m$. Introducing another indeterminate $t$
over $\Q$ which commutes with $a$ and $b$ this could equivalently be
expressed in the algebra $\Qab[t]$ as $X_m(ta,tb) = \cO(t^m)$.
This observation will play a prominent role below.

Finally, we note that when comparing with \britz\ one needs to take
into account that we have a different sign convention due to the fact
that we consider $\detm(I+A)$ instead of $\detm(I-A)$.
\end{remark}

\subsection{Reduction to a combinatorial problem in the trace class case}
\label{ss:RCP}
The following reduction is the same as in the first part of the proof of
Theorem 3.1 in \britz. To be self-contained we briefly explain this here.
Thereafter we will be able to formulate the result for which we will be
giving two independent proofs further below.

\subsubsection{First reduction to the trace class case}\label{sss:1}
The first reduction is rather trivial. Namely, suppose the product formula
in Theorem \ref{thm:main} is proved for $A,B$ of trace class.  Since the
monomials occurring in $X_m(A,B)$ are at least of degree $m$ it follows
that $\Trb{ X_m(A,B) }$ is a continuous function on $\BH m\times \BH m$.
Consequently, both sides of the product formula depend continuously on
$A,B\in\BH m$. Since $\bone$ is dense in $\BH m$ it therefore suffices to
prove the product formula for $A,B\in\bone$ of trace class.

\subsubsection{Second reduction}\label{sss:2}
Thus suppose that $A,B\in\bone$. Then by \eqref{eq:I3} and 
\cite[Theorem 6.2]{Sim77}
\begin{equation*}
    \detm(I+A)=\detF(I+A)\cdot 
       \exp\Tr\Bl \sum_{j=1}^{m-1}\frac{(-A)^j}{j}\Br,
\end{equation*}
which implies that
\newcommand{\Xtilde}{\widetilde X}
\begin{equation} \label{eq:I9}
  \begin{split}
    \detm\bl (I+A)&\cdot (I+B) \br = \detm\bl I+A+B+AB \br \\
                  &= \detm (I+A) \cdot \detm (I+B) \cdot \exp\Trb{ \Xtilde_m(A,B) },
  \end{split}
\end{equation}
with, see \eqref{eq:I7},
\begin{equation}\label{eq:I10}
  \begin{split}
    \Xtilde_m(A,B) &= 
           \sum_{j=1}^{m-1}\frac{(-1)^j}{j}\bl(A+B+AB)^j-A^j-B^j\br \\
        &= \sum_{j=1}^{m-1}\frac{(-1)^j}{j}\Bl
       \sum_{\xi\subset\setj}  \yaj(A,B) - A^j -B^j\Br.
  \end{split}
\end{equation}
Hence for $A,B$ of trace class it is rather simple to derive a product
formula (i.~e. \eqref{eq:I9}) and hence an explicit candidate for the
correction factor on the right hand side of \eqref{eq:main}.
The problem, however, is that if $A,B\in\BH m$ then $\Xtilde_m(A,B)$
is not even well-defined as it contains monomials of degree $<m$.
Those monomials are not of trace class if $A,B$ are only in $\BH m$.

The heart of the proof of the product formula therefore is the following
Theorem.

\begin{theorem}\label{thm:main1}
\simpleenum{1.} For $A,B\in\bone$ we have
\begin{equation}\label{eq:I12}
        \Trb{ X_m(A,B) }  = \Trb{ \Xtilde_m(A,B) } .
\end{equation}
\simpleenum{2.} In the free polynomial $\Q$-algebra $\Qab$ generated by two
noncommuting indeterminates $a,b$ the difference
$\Xtilde_m(a,b) - X_m(a,b)$ is a sum of commutators.
\end{theorem}
\eqref{eq:I9}, \eqref{eq:I12} and the reduction to the trace class case
Section \ref{sss:1} imply Theorem \ref{thm:main} (and Proposition
\ref{p:2}). Clearly, the second part of Theorem \ref{thm:main1} implies the
first. Theorem \ref{thm:main1} is implicit in \cite[Lemma 2.4]{BCG21}. 

This is the point where our exposition deviates from the source. As
announced before we will present two independent proofs of Theorem
\ref{thm:main1} and hence of the Product Formula Theorem \ref{thm:main}.

For future reference we put 
$Z_m(t a,t b) = \Xtilde_m(t a, tb) - X_m(t a,t b)$,
explicitly
\footnote{Recall from the Remark above that we work in $\Qab[t]$,
where $t$ is an auxiliary indeterminate commuting with $a$ and $b$.}
\begin{equation}\label{eq:I13}
  \begin{split}
   Z_m&(t a,t b)  = \Xtilde_m(t a, tb) - X_m(t a,t b) \\
     & = \sum_{j=1}^{m-1}\Biggl[
        \frac{(-t)^j}{j}\Bl(a+b+t ab)^j-a^j-b^j 
           -\sum_{\substack{
               \xi\subset\setj\\
                j+|\xi|\geq m}
            }     \yaj(a,b)\, t^{|\xi|}\Br
            \Biggr]\\
        &= \sum_{j=1}^{m-1}\frac{(-t)^j}{j}\Bl
       \sum_{\substack{
               \xi\subset\setj\\
                j+|\xi|< m}
            }
        \yaj(a,b)\, t^{|\xi|} -a^j -b^j\Br,
    \end{split}
\end{equation}
where again \eqref{eq:I7} was used.
The $t$ is just a convenient marker to count orders.

\subsection{The multi-factor case}
\label{ss:MFC}
The generalization to several factors of Theorem \ref{thm:main}
now reads as follows
\begin{theorem}\label{thm:main-mf}
There exist explicit polynomials $\Xmr\avec$
in the free polynomial algebra $\Qar$ over $\Q$ in $r$ noncommuting
indeterminates $\avec$ such that the following holds:
\begin{thmenum}
\item $\Xmrr$ is of degree at most $r\cdot (m-1)$. 
\item All monomials occuring in $\Xmrr$ have degree at least $m$.
\item Let $\cH$ be a separable Hilbert space and let 
$\Avec\in \BH m, m\in\N, m\ge 1$, be $m$-summable operators. Then 
$\Xmr\Avec$ is of trace class and one has
\begin{equation}\label{eq:main-mf}
  \detm\Bl \prod_{l=1}^r (I+A_l) \Br
    = \left( \prod_{l=1}^r \detm (I+A_l) \right)
       \cdot \exp\Trb{\Xmr\Avec}.
\end{equation}
\end{thmenum}
\end{theorem}
\begin{remark}
\simpleenum{1. } We learned from Rupert Frank that in
\cite[Lemma C.1]{Fra18} he proves a formula for three factors
if one of the operators is of finite rank. He kindly suggested to
us to consider the multi-factor case which is gratefully acknowledged.

\simpleenum{2. }
We emphasize that the second condition is crucial for $\Xmr\Avec$
being of trace class. 

\simpleenum{3. }
The explicit formula for $\Xmrr$ will be given after some preparations in
\eqref{eq:4.6} in Section \ref{s:ACP} in the text below.
\simpleenum{4. } The proof in the multi-factor case requires a bit
more algebraic and combinatorial preparations which make the proof
appear unduly lengthy. We emphasize, however, that both the algebraic
as well as the combinatorial facts outlined in some detail are very
elementary and that the basic idea of proof is very simple as the
short final section shows.
\end{remark}

\subsection{} The paper is organized as follows. In the first part
we present the two-factor case. This is basically the first version 
of this paper which was put on \verb+arXiv+ in Feb. 2022. 
In this part we give two independent
proofs of the product formula, one using a tiny bit of complex 
analysis and the other being completely algebraic.

In the second part we present the multi-factor case. This requires
a bit more algebraic preparations which are given in Section \ref{s:ACP}.
Section \ref{s:RCPTCC} then reduces the problem to a combinatorial
problem for trace class operators. Finally, in Section \ref{s:TPT}
we present again two proofs of the multi-factor case. The two
proofs are formally analogues of the two proofs given in the 
two-factor case. However, both proofs are formulated in the language
of formal power series without reference to complex analysis.

In principle the two parts are independent and could have been either
presented as paper I and paper II or they could have been more streamlined
into one slightly shorter paper dealing with the most general case only.
However, we wanted to keep the original flavor and therefore chose
the current presentation.

\subsection{Acknowledgment} The first part of this
paper is based on the first named author's Master's thesis \cite{KK21}
written under the supervision of the second named author within the
Master's program of the Mathematical Institute of the University of Bonn. A
first version on the two-factor case was put on \verb+arxiv+ as
\verb+arXiv:2202.12923v1+.  To consider the multi-factor case was suggested
to us by Rupert Frank whose insight and interest is appreciated.

\clearpage
\part{The case of 2 factors}\label{part1}
\newenvironment{msummary}{%
\let\origabstractname\abstractname
\renewcommand{\abstractname}{Summary}
\begin{abstract}}{%
\end{abstract}
\let\abstractname\origabstractname}

\begin{msummary}
In this first part we present our two new proofs of Theorem \ref{thm:main},
cf. \cite[Theorem 1.1]{BCG21}.
\end{msummary}


\section{An analytic proof of the product pormula}
\label{s:APPF}

\subsection{Preliminaries}\label{ss:P}
The following Lemma can be found between the lines in \cite{Sim77} and 
it can even be traced back to Poincare 
\cite[Second paragraph in Sec. 6]{Sim77}.
\begin{lemma}\label{lem:analytic}
Let $A\in\BH m$ with $\|A\|<1$. Then the series
$\sum\limits_{j=m}^\infty \frac{(-A)^j}{j}$ converges in the trace norm
and
\begin{equation}\label{eq:2.1}
    \detm( I+ A ) = \exp\Bl -\sum_{j=m}^\infty \frac{ \Tr(-A)^j}{j} \Br.
\end{equation}
Consequently, $A\mapsto \sum\limits_{j=m}^\infty \frac{(-A)^j}{j}$ is an
analytic mapping
\[
    \bigsetdef{A\in\BH m}{\|A\|<1} \to \bone.
\]
Furthermore, for any $A\in\BH m$ the entire function 
$z\mapsto \detm(I+z\cdot A)$ satisfies
\begin{equation}\label{eq:2.1a}
   \detm(I+z\cdot A) = 1 +\cO(z^m),\quad\text{as } z\to 0. 
\end{equation}
That is the value $1$ at $z=0$ is assumed with multiplicity
$\geq m$.
\end{lemma} 
\begin{proof} This follows immediately from the inequality
\[
   |\Tr (-A)^j|\le \|A^j\|_1\le \|A^m\|_1 \cdot \|A\|^{j-m},
   \quad\text{for } j\ge m,
\]
and, cf. \eqref{eq:I4},
\[
\begin{split}
    (I+A) \cdot \exp\Bl\sum_{j=1}^{m-1}\frac{(-A)^j}{j}\Br
       &= \exp\Bl \sum_{j=1}^\infty\frac{ -(-A)^j}{j}\Br \cdot
                    \exp\Bl\sum_{j=1}^{m-1}\frac{(-A)^j}{j}\Br\\
       & = \exp\Bl \sum_{j=m}^\infty\frac{ -(-A)^j}{j}\Br.
\end{split}
\]
Clearly, the formula now gives (again using \cite[Theorem 6.2]{Sim77}) for
arbitrary $A\in\BH m$ and $z$ in a sufficiently small neighborhood of $0$
(i.~e. such that $ |z| \|A\|<1$)
\begin{equation}\label{eq:2.2}
    \detm( I+ z\cdot A )
      = \exp\Bl -\sum_{j=m}^\infty \frac{ \Tr(-A)^j}{j}  z^j\Br
      = 1 + \cO(z^m), \text{ as } z\to 0.
\end{equation}
As remarked in the Introduction after \eqref{eq:canonical}, $z\mapsto
\detm(I+ z\cdot A)$ is a canonical Weierstraß product of genus $m$, hence
entire. 
\end{proof}

\subsection{Proof of Theorem \ref{thm:main1} \textbf{1.}}
\label{ss:PTmain1}
It now follows from Lemma \ref{lem:analytic} that, replacing
$A$ by $z\cdot A$, and $B$ by $z\cdot B$, the left
hand side of \eqref{eq:I9} satisfies
\[
    \detm\bl (I+ z\cdot A)\cdot (I+z\cdot B) \br 
       = 1 + \cO(z^m), \text{ as } z\to 0,
\]
and the first two-factors on the right hand side of \eqref{eq:I9} satisfy
\[
   \detm(I+z\cdot A) = 1+\cO(z^m),\quad
   \detm(I+z\cdot B)=1+\cO(z^m)
\]
as well. Thus
\[
    \exp\Bl\Trb{ \Xtilde_m(zA,zB) } \Br= 1 + \cO(z^m), \text{ as } z\to 0.
\]
Taking the principal branch of $\log$ on both sides we find
\[
    \Trb{ \Xtilde_m(zA,zB) }= \cO(z^m), \text{ as } z\to 0,
\]
and further, since each monomial in $X_m(zA,zB)$ has degree
$\ge m$ we have 
\[
  \Trb{ X_m(zA,zB) } = \cO(z^m)
\]
on the nose. Taking differences we therefore find (see \eqref{eq:I13})
\[
    \Trb{ Z_m(zA,zB) }= \cO(z^m), \text{ as } z\to 0.
\]
But $\Trb{ Z_m(z A, z B) }$ is a \emph{polynomial of degree $\leq m-1$} in $z$.
Being $\cO(z^m)$ means it must vanish. This proves Theorem \ref{thm:main1}
1.  and hence the first proof of Theorem \ref{thm:main} is finished.\qed
\begin{remark} We emphasize that the decisive (easy) fact which is needed
for proving the product formula is that
$\detm(I+z\cdot A) =1+ \cO(z^m)$ or equivalently that
\begin{equation}\label{eq:2.3}
\log\detm(I+z\cdot A) = \cO(z^m), \text{ as } z\to 0.
\end{equation}
\end{remark}

\section{Algebraic approach to the product formula}
\label{s:AAPF}
%
\subsection{Elementary commutative algebra considerations}
\label{ss:ECA}
We elaborate here a bit more than would be absolutely necessary
to prove the second part of Theorem \ref{thm:main1}. Rather 
we indulge a little into elementary commutative algebra considerations.

Let $\K$ be a field of characteristic $0$ and let $\cA$ be a unital
$\K$-algebra.  By $\CA:=\Acomm \subset \cA$ we denote the commutator
\emph{subspace}, i.e. the linear span of commutators $[a,b]:=ab-ba,
a,b\in\cA$.  $\Acomm$ is a module over the \emph{center} $\ZA$ of $\cA$.

By $\trA$ we denote the quotient map (the ``trace'')
\[
   \trA:\cA \longrightarrow \cA/\CA =: \QA.
\]
A priori $\QA$ is only a \emph{$\K$-vector space}.  
E.~g. in the case of $\cA=\bone$ we have the commutative diagram
\begin{equation}\label{eq:3.1}
\begin{tikzcd}
    \cA\arrow[dr,"\Tr"] \arrow[r,"\trA"] & \QA\arrow[d, "\widetilde \Tr"] \\
                       &   \C.
\end{tikzcd}
\end{equation}

\newcommand\Dtilde{\widetilde D}
\begin{lemma}\label{lem:alg1}
\simpleenum{1.} $\Acomm$ and $\QA$ are natural
$\ZA$-modules and $\trA$ is a $\ZA$-module map
satisfying  $\trA( a\cdot b) = \trA( b\cdot a)$ for all $a,b\in\cA$.
\simpleenum{2.} If $D$ is a derivation of $\cA$ then $D$
descends to a $\ZA$-derivation
\footnote{More precisely,
$\Dtilde$ is a linear endomorphism of the $\ZA$-module
$\QA$ satisfying $\Dtilde(z\cdot a) = D(z)\cdot a+ z\cdot \Dtilde a$
for $z\in\ZA$ and $a\in\QA$.} $\Dtilde$ of $\QA$
such that the following diagram commutes
\[
\begin{tikzcd}
     \cA \arrow[d,"\trA"]\arrow[r, "D"] & \cA\arrow[d,"\trA"]\\
     \QA \arrow[r,"\Dtilde"]        & \QA.
\end{tikzcd}
\]
\end{lemma}
\begin{proof}\simpleenum{1.} follows immediately from
$z[a,b]=[za,b]=[a,zb]$ for $a,b\in\cA$ and $z\in\ZA$.
Clearly, $\trA(a\cdot b) =\trA(b\cdot a)$ holds by construction.
\simpleenum{2.} $D[a,b]= [Da,b]+[a,Db]$, hence $D$ leaves
$\Acomm$ invariant. Furthermore, for $z\in\ZA$ and $a\in\cA$
\[
    (Dz)a = D(za) - z Da = D(az) - (Da)z = a Dz,
\]
hence $D$ leaves $\ZA$ invariant as well. Consequently,
$\Dtilde$ is a well-defined $\ZA$-deriva\-tion on the quotient
$\QA$.
\end{proof}

For a polynomial $f(t)\in\K[t]$ and $a\in\cA$ we denote by
$f(a)\in\cA$ the element obtained by replacing $t$ by $a$
(the ``insertion homomorphism''). Similarly, if $f(t),g(t)\in\K[t]$
and $g(a)$ is invertible in $\cA$ then 
\[
(f/g)(a):= f(a) g(a)\ii =g(a)\ii f(a).
\]
If $\widetilde f(t),\widetilde g(t)\in\K[t]$
with $\widetilde g(a)$ invertible and $(f/g)(t)=(\widetilde f/\widetilde
g)(t)$ then indeed $(f/g)(a)=(\widetilde f/\widetilde g)(a)$.
By $\partial_t$ we denote the usual derivation on polynomial algebras
with indeterminate $t$ obtained by formal differentiation by $t$.
For $f(t)\in\K[t]$ we also write $f'(t)$ for $\partial_t f(t)$.

\begin{lemma}\label{lem:alg2} Let $a,x\in\cA$ with $[a,x]=0$ and
$f(t),g(t)\in\K[t]$ with $g(a)$ invertible. Then
\[
    \trA\bl x\cdot D\bl \frac fg(a) \br \br = 
         \trA\bl x\cdot \bl \frac fg\br'(a)\cdot Da \br,
\]
in particular 
\[
    \Dtilde \trA\bl \frac fg(a)       \br
        =   \trA\bl \bl\frac fg\br'(a) \cdot Da      \br.
\]
\end{lemma}
\begin{proof}
For a monomial $a^n$ we have
\[
     x\cdot D(a^n) = \sum_{j=0}^{n-1} x\cdot a^j\cdot (Da)\cdot a^{n-1-j},
\]
thus since $x$ commutes with $a$
\[
   \trA\bl x\cdot D(a^n) \br = \trA \bl x\cdot n a^{n-1}\cdot Da\br,
\]   
thus the first claim for monomials. By linearity it follows for
polynomials $f(t)\in\K[t]$. For $g(t)\in\K[t]$ with $g(a)$ invertible
we infer
\footnote{The Leibniz-rule implies $D(1_\cA) = D(1_\cA) + D(1_\cA)=0$.}
from applying $D$ to $1_\cA = g(a) \cdot g(a)\ii$ that 
$D\bl g(a)\ii\br = - g(a)\ii \cdot (D g(a) ) \cdot g(a)\ii$
hence
\[
\begin{split}
    \trA\Bl &x\cdot D\bl \frac fg(a) \br \Br \\
       &= \trA\Bl x\cdot \Bl D(f(a))\cdot g(a)\ii 
                   - f(a)g(a)\ii \cdot D(g(a))\cdot g(a)\ii \Br\Br \\
       &= \trA\Bl x \Bl g(a)\ii \cdot D(f(a)) 
           - f(a)g(a)^{-2} \cdot D(g(a)) \Br \Br.
\end{split}
\]
Since $xg(a)\ii$ and $x f(a) g(a)^{-2}$ commute with $a$ we may
apply the proven part to the polynomials $f(t),g(t)$ and obtain further
\[
\begin{split}
    \ldots & = \trA\Bl x\cdot \Bl g(a)\ii f'(a) - f(a) g(a)^{-2} g'(a)\Br 
                \cdot Da\Br \\
           & = \trA\Bl x\cdot \bl \frac fg\br'(a)\cdot Da \Br.
\end{split}
\]
The second claim now follows with $x=1_\cA$ together with the commutative
diagram in Lemma \ref{lem:alg1}.
\end{proof}

\subsection{Application to the product formula}
\label{ss:APF}
We now apply the previous considerations to the free unital polynomial
$\K$-algebra $\cA=\PFree{\K}{a,b}$ on two noncommuting variables $a,b$. Let
further $t$ be an indeterminate commuting with $a,b$ and consider the
polynomial algebra $\cA[t]=\PFree{\K}{a,b}[t]$. On $\cA[t]$ we have the
derivation $D:=\frac{d}{dt}$ sending $t^k$ to $k t^{k-1}$. Furthermore, we
fix $m\in\N$, $m\ge 1$. Note that for $x\in\cA$ the polynomial $1+x t$ is
invertible in the quotient $\cA[t]/(t^m)$ with inverse
\footnote{Alternatively, one could work in the formal power series
in $t$ with coefficients in $\cA$
where $1+x t$ is invertible on the nose, cf. Part \ref{part2}. }

\[
    (1+ x t +\cO(t^m))\ii  = \sum_{j=0}^{m-1} (-1)^j x^j t^j +\cO(t^m),
\]
where $\xi+\cO(t^m)$ stands for the class of $\xi\mod t^m$.

As before we denote by $\trCt:\cA[t] \to \Qcsp{\cA[t]}$ the quotient map. It
is elementary to check that the spaces $\Qcsp{\cA[t]}$ and $\Qcsp\cA[t]$ 
are isomorphic (as $\K$-vector spaces resp. $\ZA$-modules)
via the correspondence 
\begin{equation} \label{iso}
    \trCt(a_0+a_1t+...+a_nt^n)\leftrightarrow 
    \trC(a_0)+\trC(a_1)t+\ldots+\trC(a_n) t^n.
\end{equation}
Recall from Proposition \ref{p:2} and Section \ref{ss:RCP} the
definition of $X_m(ta,tb)$, $\Xtilde_m(ta,tb)$, and $Z_m(ta,tb)$. From the
isomorphism \eqref{iso} and the fact that $X_m(ta,tb)=\cO(t^m)$ the claim
\textbf{2.} in Theorem \ref{thm:main1} is equivalent to 
\begin{equation}\label{eq:equ1}
   \trCt\Bl \Xtilde_m(ta,tb) \Br = \cO(t^m),
\end{equation}
meaning that the left hand side is a polynomial in $t$ which is divisible
by $t^m$. Since $\Xtilde(ta,tb)$ has no constant term
\eqref{eq:equ1} is equivalent to
\begin{equation}\label{eq:equ2}
    \partial_t \trCt\Bl \Xtilde_m(ta,tb) \Br = \cO(t^{m-1}).
\end{equation}
To see this we introduce the polynomial
\[
    f(x):=\sum_{j=1}^{m-1}\frac{(-1)^j}{j}x^j \in \K[x].
\]
Clearly $\mod x^{m-1}$ we find
\[
    f'(x) +\cO(x^{m-1}) = -\sum_{j=0}^{m-2} (-x)^j +\cO(x^{m-1})
        = -\bl1+x+\cO(x^{m-1})\br\ii.
\]
By slight abuse of notation, to save space, we will write
$(1+x)\ii \mod x^{m-1}$ for $(1+x+\cO(x^{m-1}))\ii$.
With $f$ we have 
\[
      \Xtilde_m(ta,tb) = f( ta+tb+t^2ab) - f(ta) -f(tb),
\]
and applying Lemma \ref{lem:alg2} we find
\begin{multline}\label{eq:3}
    \partial_t \trCt\bl \Xtilde_m(ta,tb) \br \mod t^{m-1}\\
      =       \trCt\bl f'(ta+tb+t^2ab)(a+b+2tab)-f'(ta)a-f'(tb)b\br.
\end{multline}
We note that
\begin{align*}
    f'(ta+tb+t^2ab) \mod t^{m-1} & = - (1+tb)\ii(1+ta)\ii \mod t^{m-1}\\
    a+b+2tab        & = (1+ta)b+a(1+tb),\\
    f'(ta) a  \mod t^{m-1}  & = -(1+ta)\ii a \mod t^{m-1},\\      
    f'(tb) b  \mod t^{m-1}  & = -(1+tb)\ii b \mod t^{m-1}.      
\end{align*}
Plugging this into the previous expression \eqref{eq:3} and exploiting that
$\trCt$ vanishes on commutators implies 
$\partial_t \trCt\bl\Xtilde_m(ta,tb) \br \mod t^{m-1}=0$ and hence the claim.

\subsection{Comparison to \britz}
\label{ss:CB}
We briefly discuss how the decisive combinatorial Lemma \cite[Lemma
2.3]{BCG21}, which a priori seems more general than Theorem \ref{thm:main1}
2., can be obtained quite easily from our method as well.

Denote by $\Pi_j$ the set of partitions $\pi=(\pi_1,\pi_2,\pi_3)$ of the
set $\{1,\ldots,j\}$ into three subsets and by $\Pi_{j,k_1,k_2}$ the set of
those $\pi=(\pi_1,\pi_2,\pi_3)\in\Pi_j$ with $|\pi_1|+|\pi_3|=k_1$ and
$|\pi_2|+|\pi_3|=k_2$. Furthermore put
\begin{equation}\label{eq:C1}
    z_{\pi,k}(a,b):=\begin{cases} a,  & k\in \pi_1, \\
                                  b,  & k\in \pi_2, \\
                                 ab,  & k\in \pi_3,
           \end{cases},\quad 
             z_\pi(a,b):=\prod_{k=1}^j z_{\pi,k}(a,b).
\end{equation}
$\pi_3$ plays the role of $\xi$ in \eqref{eq:I6}. With this
notation one has
\[
      \yaj(a,b) = \sum_{\substack{
                       \pi\in\Pi_j\\
                       \pi_3 = \xi}} z_\pi(a,b)
\]
and further, see \eqref{eq:I13}
\begin{equation}
\begin{split}
    Z_m(a,b) &= \sum_{j=1}^{m-1}\frac{(-1)^j}{j}\Bl
       \sum_{\substack{
               \pi\in\Pi_j\\
                j+|\pi_3|< m}
            }
          z_\pi(a,b)\, -a^j -b^j\Br,\\
    & = \sum_{\substack{
         k_1,k_2> 0\\ k_1+k_2 <m}}
         \underbrace{
         \sum_{j=1}^{k_1+k_2} \frac{(-1)^j}{j}
             \sum_{\pi\in\Pi_{j,k_1,k_2}}
                 z_\pi(a,b)
        }_{=: z_{k_1,k_2}(a,b)}.
    \end{split}
\end{equation}
Note that
\[
z_{0,0}(a,b) = 0,\quad
z_{0,k_2}(a,b)= \frac{(-b)^{k_2}}{k_2}, k_2>0,
\quad
z_{k_1,0}(a,b)= \frac{(-a)^{k_1}}{k_1}, k_1>0.
\]
This explains the conditions $k_1,k_2>0$ in the last formula.

\subsection{}\label{ss:3.4}
Lemma \cite[Lemma 2.3]{BCG21} now states that 
$z_{k_1,k_2}$ is a sum of commutators for $k_1,k_2\geq 1$.
This is seemingly stronger than the corresponding statement for
$Z_m(a,b)$. However, introducing the commuting indeterminates $s,t$
commuting with $a,b$ we have in the algebra $\PFree{\K}{a,b}[s,t]$
\[
Z_m(sa,tb) = 
    \sum_{\substack{
         k_1,k_2> 0\\ k_1+k_2 <m}
         }
        z_{k_1,k_2}(a,b) \cdot s^{k_1} t^{k_2}
\]
and hence by the proven Theorem \ref{thm:main1}
(inserting $s a$ for $a$ and $tb$ for $b$)
and by \eqref{iso}
\[
\begin{split}
   0 &= \tr_{\PFree{\K}{a,b}[s,t]}\Bl Z_m(sa,tb) \Br\\
     & = \sum_{\substack{
         k_1,k_2> 0\\ k_1+k_2 <m}
         }
    \tr_{\PFree{\K}{a,b}}\bl z_{k_1,k_2}(a,b)\br
        \cdot s^{k_1} t^{k_2}
\end{split}
\]
and hence all $\tr_{\PFree{\K}{a,b}}\bl z_{k_1,k_2}(a,b)\br$
must vanish for $k_1,k_2>0, k_1+k_2<m$.

\clearpage
\part{Generalization to the multi-factor case}\label{part2}
\begin{msummary}
Here we present the multi-factor case. At the same time
we streamline further the algebraic as well as the combinatorial
aspects of the problem. 
\end{msummary}
\section{Algebraic and combinatorial preparations}
\label{s:ACP}

\renewcommand{\thesubsubsection}{\arabic{subsubsection}}
\subsection{The free polynomial algebra in $r$ noncommuting variables}
\label{ss:FPA}
 
We discuss here results analogous to Section \ref{ss:ECA} in the context
of formal power series. For general facts on formal power series, see
\cite[IV.9]{Lan02}.

Let $\K$ be a field of characteristic $0$ and let $\KFree:=\KFreeExpanded$
be the free polynomial algebra in the $r$ noncommuting indeterminates $\avec$.
The monomials are words in the alphabet $\{\avec\}$, i.e. a priori there
are no relations between the $\avec$.

As in Section \ref{ss:ECA} $\QA$ denotes the quotient of $\cA$ by the
commutator subspace $\CA=[\cA,\cA]$.

In the sequel we will need to adjoin \emph{central} indeterminates.  A
single one will usually be denoted by $x$. The letter $t$ is reserved for
an $r$-tuple $(\tvec)$ of indeterminates. In order to avoid unnecessary
repetitions, we will mostly work with $t$ with the understanding that the
results also hold for a single indeterminate $x$.
For $t$ we will also make use of multiindex notation,
$t^\ga$ then stands for $t^{\ga_1}\cldots t_r^{\ga_r}$; in this
context multiindices will be denoted by greek letters.

Let then $t:=(\tvec)$ be \emph{central} indeterminates. I.e. the
$\tvec$ commute and the $\tvec$ commute with the $\avec$.
$\KFree[\tvec]=:\KFree[t]$ then denotes the usual polynomial algebra over
$\KFree$ obtained by adjoining $t=(\tvec)$. 

Similarly, $\KFormal \tvec=:\Att$ denote the formal power series in
$\tvec$ over $\KFree$, i.e. formal sums of the form
$\sum\limits_{\ga\in\N^r}f_\ga\cdot t^\ga$ with $f=g$ iff
all coefficients $f_\ga=g_\ga$ are equal. $\Att$ is an algebra
with Cauchy-product
\begin{equation}\label{eq:Cauchy}
   f(t) \cdot g(t) = \sum_{\ga} \Bl \sum_{\gb\le \ga} f_\gb \cdot 
       g_{\gb-\ga} \Br \cdot t^\gamma,
\end{equation}
where $\gb\le \ga$ means $\ga-\gb\in\N^r$.
The center $\zen{\Bl \KFormal t \Br}$ \emph{equals}
$\FPower{\zen\KFree} t$. The $\ZAtt$-modules $\QAtt$ are
defined accordingly. Indeed, in view of Lemma \ref{lem:alg1} 
the Cauchy-product \eqref{eq:Cauchy} of formal power series turns $\QAtt$ into
a $\ZAtt$-module.

\newcommand{\plj}[1][j]{\frac{\partial}{\partial t_{#1} }}
For any vector space $\cB$ ($\cB$ stands for $\cA, \ZA, \QA$ etc.)
with corresponding formal power series vector space $\FPower\cB t$ the formal
partial derivative $\plj$ acts as a linear map reducing the degree
by one. $\plj$ satisfies the Leibniz-rule $\plj(f\cdot g) =
(\plj(f))\cdot g + f\cdot \plj g$, whenever the product $f\cdot g$
makes sense (e.g. when $\cB$ is an algebra, or for $f\in\ZAtt$ and
$g\in\QAtt$). 

Furthermore, we denote by $\FPower{\cB_N}t$ the set of those
$f=\sum\limits_\ga f_\ga\cdot t^\ga\in\FPower{\cB} t$ 
with $f_\ga=0$ for $|\ga|<N$, i.e. $f$ takes the form
$
     f = \sum\limits_{|\ga|\ge N} f_\ga \cdot t^\ga.
$

Clearly, $\Attn N$ is an ideal in $\Att$ and more generally
\[
     \Attn N \cdot \Attn M \subset \Attn {N+M},\quad
     \ZAttn N \cdot \QAttn M \subset \QAttn {N+M},
\]
in particular $\QAttn N$ is a $\ZAtt$-module.

Recall from Section \ref{ss:ECA} the facts about the trace map $\trA$.  In
contrast to \eqref{iso}, we may not expect $\Qcs{\bl \Att \br }$ to be
isomorphic to $\QAtt$. We circumvent this by working solely with $\QAtt$,
the formal power series over the quotient $\QA$, and define the map
\begin{equation}\label{sigma-map}
    \trAt: \Att \longrightarrow \QAtt, \sum_\ga a_\ga\cdot t^\ga \mapsto \sum_\ga
    \trA(a_\ga)\cdot t^\ga.
\end{equation}

\begin{lemma}\label{lem:alg3}
\simpleenum{1.} $\trAt$ is a well-defined, surjective $\ZAtt$-module
map.

\simpleenum{2.} $\trAt$ satisfies $\trAt(f\cdot g) = \trAt(g\cdot f)$
for $f,g\in\Att$.

\simpleenum{3.} The kernel of $\trAt$ equals $\CAtt$, that is if
$f=\sum\limits_\ga f_\ga t^\ga \in\Att$ with $\trAt(f)=0$ then
each $f_\ga \in\CA$ is a sum of commutators.

\simpleenum{4.} The formal partial derivative $\plj$
commutes with $\trAt$, i.e. the following diagram commutes
\[
\begin{tikzcd}
     \Att \arrow[d,"\trAt"]\arrow[r, "\plj"] & \cA\arrow[d,"\trAt"]\\
     \QAtt \arrow[r,"\plj"]        & \QAtt.
\end{tikzcd}
\]

\end{lemma}
\begin{proof} \textbf{1.} is straightforward and \textbf{2.}
follows by applying $\trA( a\cdot b) = \trA( b\cdot a)$ term by term.
\simpleenum{3.} is equally obvious as $\trAt(f)=0$ implies that 
$\trA( f_\ga ) =0$ for all $\ga$ and hence the claim is just a reformulation
of the latter.

Finally, \textbf{4. } follows from a term by term check.
\end{proof}

We now have the following formal power series analogue of Lemma
\ref{lem:alg2}:

\begin{lemma}\label{lem:alg4}  Let $f\in\FPower \K x$ and let 
$g\in\Attn 1$ be without constant term. Denote
by $f(g(t))\in\Att$ the element obtained by inserting $g$ into $f$,
which is well-defined thanks to $g(0)=0$. Then we have for the
formal partial derivative $\plj$ and the ``trace'' $\trAt$
\[
   \plj\trAt\Bl f\bl g(t)\br \Br = 
       \trAt\Bl f'\bl g(t) \br \cdot \plj g(t)  \Br.
\]
\end{lemma}
\begin{proof} \simpleenum{1. } Note first, that the claim is true
if $f(x) = x^n$ is a monomial. This follows already from Lemma
\ref{lem:alg2} applied with the $\cA$ there being $\Att$,
the $D$ there being $\plj$, and the $a$ there being $g(t)$;
alternatively
\[
  \plj g(t)^n = \sum_{j=0}^{n-1} g(t)^{j}\cdot \bl \plj g(t) \br\cdot g(t)^{n-1-j},
\]
hence after taking $\trAt$, and using Lemma \ref{lem:alg3},
\[
   \plj \trAt\Bl g(t)^n \Br = \trAt\Bl n\cdot g(t)^{n-1}\cdot \plj g(t) \Br
           =\trAt\Bl f'\bl g(t) \br \cdot \plj g(t) \Br
\]
in this case.

\simpleenum{2. } Given a general $f\in\FPower \K x$ write for arbitrary $N$
\[
   f(x) = \sum_{n=0}^N f_n \cdot x^n + x^{N+1}\cdot R_N(x)
\]
with $R_N\in\FPower \K x$. Then
\[
    f\bl g(t) \br = \sum_{n=0}^N f_n \cdot g(t)^n + g(t)^{N+1}\cdot R_N\bl
    g(t)\br.
\]
The remainder $g(t)^{N+1}\cdot R_N\bl g(t) \br$ lies in the ideal $\Attn {N+1}$
and hence we have by the first part of this proof
\[
\begin{split}
   \plj\trAt\Bl f\bl g(t) \br \Br 
        &= \sum_{n=0}^N n\cdot f_n\cdot \trAt\Bl g(t)^{n-1} \cdot \plj g(t) \Br \mod \QAttn {N} \\
        &= \trAt\Bl f'\bl g(t) \br \cdot \plj g(t) \Br \mod \QAttn {N}.
\end{split}
\]
Since $N$ is arbitrary and since $\bigcap\limits_N \QAttn{N}=\{0\}$ the claim follows.
\end{proof}

\subsection{Combinatorial preparations}
\subsubsection{} Recall from \ref{ss:Not}  that for noncommuting variables,
product notation means that the product is taken in the given order, 
i.e.  $\prod_{l=1}^N x_j = x_1\cldots x_N$ in this order. 

For a subset $\xi\subset\setr$ with
$\xi = \{l_1,\ldots,l_p\}$, $l_1<l_2<\ldots<l_p$, put
\begin{equation}\label{eq:41}
        z_\xi(\avec):= \prod_{k=1}^p a_{l_k},\quad
        z_{\emptyset}(\avec) := 1.
\end{equation}
With this notation we have
 $\displaystyle\mpi l r (1+a_l) = \sum\limits_{\xi\subset\setr} z_\xi$.

\subsubsection{} Denote by $\mfMjr$ the set of maps
$f:\setj \to \mfP(\setr)\setminus \{\emptyset\}$, where
$\mfP(\setr)$ denotes the power set of $\setr$. For
$f\in\mfMjr$ let
\begin{equation*}
     z_f(\avec) := \mpi k j z_{f(k)}.
\end{equation*}
Let us compare this to the $z_\pi$ of Section \ref{ss:CB} above.
There $r=2, a_1=a, a_2=b$ and to a partition
$\pi=(\pi_1,\pi_2,\pi_3)$ of the set $\{1,\ldots,j\}$
there corresponds a map $f:\{1,\ldots,j\}\to \{\{1\},\{2\},\{1,2\}\}$,
where 
\begin{equation*}
    f(k):=\begin{cases} \{1\},  & k\in \pi_1, \\
                        \{2\},  & k\in \pi_2, \\
                        \{1,2\},  & k\in \pi_3,
           \end{cases},\quad
    z_{f(k)} = \begin{cases} a, & f(k) = \{1\},\\
                             b, & f(k) = \{2\},\\
                             ab, & f(k) = \{1,2\}.
            \end{cases}
\end{equation*}
Consequently $z_\pi(a,b) = z_f(a_1,a_2)$. With this notation we now
have
\begin{equation}\label{eq:4.2}
    \Bl\prod_{l=1}^r (1+a_l) - 1 \Br^j 
  = \Bl \sum_{\emptyset\not=\xi\subset \setr} z_\xi(\avec) \Br^j
  = \sum_{ f \in\mfMjr} z_f( \avec ).
\end{equation}

\subsubsection{Orders} For a monomial $w$ (word) in the
algebra $\KFree$ we denote by $L(w)$ its \emph{length} or
\emph{order}, explicitly $w = x_1\cldots x_{L(w)}$ where
$x_l\in\{\avec\}$. Given $f\in \mfMjr$ we have 
$L(f(k)) = |f(k)| = \operatorname{cardinality}(f(k))$
and 
\begin{equation*}
   \begin{split}
    L(z_f(\avec ))& = \sum_{k=1}^j L(z_{f(k)}) = 
                           \sum_{k=1}^j | f(k) | .
    \end{split}
\end{equation*}
Note that for $f$ one clearly has 
$ j= \sum\limits_{\xi\subset\setr} |f\ii(\{\xi\})|$.

\subsubsection{} Now we define the analogues of the $z_{k_1,k_2}(a,b)$
of Section \ref{ss:CB} in the multi-factor case. For a multiindex
$\ga = (\ga_1,\ldots,\ga_r)\in \N^r$ \footnote{cf. Section
\ref{ss:Not}, recall $0\in\N$ and $|\ga|=\ga_1+\ldots+\ga_r$} 
and $f\in \mfMjr$ let $\ga(f,l)$ be the number of $a_l$ occuring
in the word $z_f(\avec)$, i.e. 
\begin{equation}\label{eq:43}
    \ga(f,l):= \sum_{l\in\xi\subset\setr} |f\ii(\xi)|
       = \sum_{\substack{k=1\\l\in f(k)}}^j 1
       = | \setdef{k}{l\in f(k)}| \le j,
\end{equation}
and further $\ga(f):=(\ga(f,1),\ldots,\ga(f,r))$.

For $f\in\mfMjr$ we have from 4. and \eqref{eq:43}
\begin{equation}\label{eq:44}
    j\cdot r \ge \ga(f,1)+\ldots+\ga(f,r) 
       = \sum_{l=1}^r \sum_{l\in\xi\subset\setr} |f\ii(\{\xi\})|
       \ge \sum_{\xi\subset\setr}|f\ii(\{\xi\})| = j.
\end{equation}

\subsection{The correction term in the multi-factor product formula}

Now we are ready to define the multivariate analogues
of the polynomials $X_m(a,b)$, $\Xtilde(a,b)$, and $Z_m(a,b)$, cf.
\eqref{eq:I8}, \eqref{eq:I10}, and \eqref{eq:I13}. Namely, put
for $a=(\avec)$
\begin{equation}\label{eq:4.5}
  \begin{split}
    \Xtmr a &= \sum_{j=1}^{m-1}\frac{(-1)^j}{j}
                  \biggl[ \Bl \mpi lr (1+a_l) - 1 \Br^j - \msum l r a_l^j
                  \biggr] \\
      &= \msum j {m-1} \frac{(-1)^j}{j}
             \biggl[   \sum_{f\in\mfMjr} z_f(a) - \msum l r a_l^j\biggr] \\
      &= \sum_{j=1}^{m-1}\frac{(-1)^j}{j}
        \sum_{ \ga\in\Nrs  }
           \sum_{\substack{ f \in \mfMjr\\ \ga(f) = \ga } } z_f(a).
  \end{split}
\end{equation}
Here, 
\begin{equation}\label{eq:Nrs}
     \Nrs := \N^r\setminus \bigsetdef{\ga\in\N^r}{ \exists_l \ga_l = |\ga| }.
\end{equation}
In words, tuples of the form $(0,\ldots,0,|\ga|,0,\ldots,0)$
are excluded from $\Nrs$. The reason for this are the easy to
verify formulas 
\begin{equation}
     \begin{split}
     z_{\ga=(0,\ldots,0,j,0,\ldots,0)} &= \frac{(-1)^{j-1}}{j} a_l^j,\\
     z_{(0,\ldots,0)}(a) &= 0,
     \end{split}
\end{equation}
where $\ga=(0,\ldots,0,j,0,\ldots,0) \in\N^r\setminus\Nrs, j>0$ 
($j$ in the $l$-th slot).

Furthermore, from the second to the third line in \eqref{eq:4.5}
we have used \eqref{eq:4.2} and the 
disjoint decomposition 
\[
    \mfMjr = 
   \bigcup\limits_{\ga\in N^r} \bigsetdef{f\in\mfMjr}{\ga(f) = \ga}.
\]

$\Xmr a$ is now obtained from $\Xtmr a$ by removing all summands
of order $<m$:
\begin{equation}\label{eq:4.6}
    \Xmr a := \sum_{j=1}^{m-1}\frac{(-1)^j}{j}
            \sum_{\ga\in\Nrs }
           \sum_{\substack{ f \in \mfMjr\\ \ga(f) = \ga } } z_f(a)
\end{equation}
and finally, cf. \eqref{eq:I13}
\newcommand{\sumalphaf}{\sum_{\substack{ f \in \mfMjr\\ \ga(f) = \ga} } }
\newcommand{\sumalpha}{\sum_{\substack{ \ga\in\Nrs \\ |\ga|<m}}}
\begin{equation}\label{eq:4.7}
\begin{split}
      Z_{m,r}(a) & = \Xtmr{a} - \Xmr{ a } = \sum_{j=1}^{m-1}\lncfpl
            \sumalpha\sumalphaf z_f(a) \\
         &= \sumalpha \msum j {|\ga|}\lncfpl
            \sumalphaf  z_f(a) 
         =: \sumalpha   z_\ga(a),
\end{split}
\end{equation}
where
\begin{equation}\label{eq:4.8}
        z_\ga(a) =  \sum_{j=1}^{|\ga|} \lncfpl  \sumalphaf z_f(a).
\end{equation}
The upper limit $|\ga|$ in the inner sum follows from \eqref{eq:44}.

\subsection{The $\log$ of the regularized determinant as
formal power series}

We return to formal power series: 
for $b\in\KFree$ and a formal power series indeterminate $z$,
cf. Section \ref{ss:FPA}, let
\begin{equation}\label{eq:4.9}
     \log (1+ z\cdot b) = \sum_{j=1}^\infty \frac{(-1)^{j-1}}{j} b^j \cdot z^j
\end{equation}
resp. in analogy to \eqref{eq:I4}, \eqref{eq:2.2} let 
\begin{equation}\label{eq:4.9.1}
    \begin{split}
    \logm(1+z\cdot b)&:= \log(1+z\cdot b) - \msum j{m-1} 
                      \lncf b^j\cdot  z^j\\
        & = \msum[m] j \infty \lncf b^j\cdot z^j,
    \end{split}
\end{equation}
as elements of the formal power series algebra $\KFormal z$.
Finally, taking the ``trace`` map $\trAt[z]$
\newcommand{\ldm}[1][\cA]{\operatorname{logdet}\limits_m^{#1}}
\newcommand{\ld}[1][\cA]{\operatorname{logdet}\limits^{#1}}
\begin{equation}\label{eq:4.9.2}
\begin{split}
    \ldm(1 + z\cdot b) &:= \trAt[z]\bl \logm(1+z\cdot b) \br \\
    \ld(1+z\cdot b) &:=\trAt[z]{\bl \log(1+z\cdot b) \br = 
    \operatorname{logdet}\limits_1^{\cA}}(1+z\cdot b).
\end{split}
\end{equation}
as an element of $\QAtt[z]$. $\ldm$ is the formal
power series analogue of the function $\log\detm( 1 + z\cdot A)$
in the operator case and with obvious identifications, 
for $A\in\BH m$ the function
$\log\detm( 1+ z\cdot A ) $ is obtained by inserting
$A$ into $\ldm( 1+ z\cdot a ) $ for $a$ and taking the Hilbert space
trace, cf. the commutative diagram \eqref{eq:3.1}. In sum, the
formal power series \eqref{eq:4.9}, \eqref{eq:4.9.1} allow for
a completely algebraic treatment of the product formula for regularized
Fredholm determinants. The next result is the formal power series analogue
of the product formula for regularized determinants.

\clearpage 
\begin{prop}\label{l:log-functional-equation}
In the formal power series algebra $\KFormal t$ resp.
in $\QAtt$ we have the following identities:
\begin{thmenum} 
\item
\begin{equation}\label{eq:4.10}
     \log \Bl \mpi l r (1 + t_l \cdot a_l) \Br 
      =  \sum_{\ga\in\N^r} z_\ga(a)\cdot t^\ga,
\end{equation}
cf.  \eqref{eq:4.8}.
\item
\begin{equation}\label{eq:411}
    \ld\Bl \mpi l r (1 + t_l \cdot a_l ) \Br 
    = \msum l r \ld ( 1 + t_l\cdot a_l )
\end{equation}
in the formal power series $\QAtt$,
\item
\begin{equation}\label{eq:4.14}
  \begin{split}
    &\ldm\Bl \mpi l r (1 + t_l \cdot a_l ) \Br  - 
                \msum l r \ldm ( 1 + t_l\cdot a_l ) \\
       & = \trAt \bl \Xmrr( t_1\cdot a_1,\ldots,t_r\cdot a_r )\br
          +\trAt \bl \Zmrr( t_1\cdot a_1,\ldots,t_r\cdot a_r
          )\br.
  \end{split}
\end{equation}
\end{thmenum}
\end{prop}
\begin{proof}
\simpleenum{1. } Using \eqref{eq:4.2} and \eqref{eq:4.5}-\eqref{eq:4.8}
we find
\begin{equation}
\begin{split}
     \log &\Bl \mpi l r (1 + t_l \cdot a_l) \Br 
       =  \sum_{j=1}^\infty \lncf
             \Bl \mpi lr (1+ t_l\cdot a_l) - 1 \Br ^j \\
       & = \sum_{j=1}^\infty \lncf\sum_{f\in\mfMjr} 
                  z_f( t_1 a_1,\ldots,t_r a_r) \\
       & = \sum_{j=1}^\infty \lncf\sum_{\ga\in\N^r} 
       \sum_{\substack{f\in\mfMjr \\ \ga(f) = \ga}} z_f( a ) t^\ga 
           =  \sum_{\ga\in\N^r} z_\ga(a)\cdot t^\ga,
\end{split}
\end{equation}

\simpleenum{2. }
The left hand side of \eqref{eq:411} equals
\[
    \trAt  \Bl \log\Bl 1+\bl \mpi l r (1 + t_l \cdot a_l ) -1 \br \Br
    \Br.
\]
To take the formal partial derivative w.r.t. $t_j$ we may therefore apply
Lemma \ref{lem:alg4} with $f(z) = \log(1+ z)$ and
$g(t) = \mpi l r (1 + t_l \cdot a_l ) -1$. We have
\[
    \plj\Bl  \mpi l r (1 + t_l \cdot a_l ) -1 \Br
        = \prod\limits_{1\le l<j} (1 + t_l\cdot a_l) \cdot a_j\cdot
          \prod\limits_{j<l\le r} (1 + t_l\cdot a_l).
\]
and
\[
   f'(g(t)) = \prod\limits_{l=r}^1 (1+ t_l\cdot a_l)\ii,
\]
hence using the tracial property of $\trAt$, Lemma \plref{lem:alg3},
\[
   \trAt\Bl f'\bl g(t) \br \cdot \plj g(t)  \Br
       = \trAt\Bl (1+ t_j\cdot a_j)\ii \cdot a_j\Br.
\]
On the other hand we find for the formal partial derivative w.r.t. $t_j$
of the right hand side of \eqref{eq:411}
\[
    \plj \trAt\Bl \log\bl 1  + t_j\cdot a_j\br \Br = 
       \trAt\Bl (1+ t_j\cdot a_j)\ii \cdot a_j\Br,
\]
again by invoking Lemma \ref{lem:alg4}. 

Hence the formal partial derivatives w.r.t. to $t_j, 1\le j\le r$,
of the left and right hand sides of \eqref{eq:411} coincide.
Since both sides have no constant term they must hence coincide as well.

\simpleenum{3. }
From the definition \eqref{eq:4.9.1}, the proven part (2) and
\eqref{eq:4.5} we see that the left hand side equals 
$\trAt \bl \Xtmrr( t_1\cdot a_1,\ldots,t_r\cdot a_r )\br$
and hence, by \eqref{eq:4.7} the claim.
\end{proof}

\section{Reduction to a combinatorial problem in the trace class case}
\label{s:RCPTCC}

Having defined the polynomials $\Xtmrr, \Xmrr, Z_{m,r}$ we now
proceed in parallel to the two-factor case and
reduce the proof of Theorem \ref{thm:main-mf} to a
purely combinatorial problem in the trace class case.

\subsection{First reduction to the trace class case}\label{ss:FRTCC}
The argument here is the same as in Section \ref{sss:1}. Namely,
since the monomials in $\Xmr\avec$ are at least of degree $m$ it
follows that $\Trb{\Xmr\Avec}$ is a continuous function
on the $r$-fold cartesian product $\bl\BH m\br^r$.
Consequently, both sides of the product formula 
\eqref{eq:main-mf} depend continuously on $\Avec$.
Since $\bone$ is dense in $\BH m$ it therefore suffices to
prove \eqref{eq:main-mf} for $\Avec\in\bone$ of trace class.

\subsection{Second reduction}\label{ss:SR}
In view of the algebraic and combinatorial preparations of Section
\ref{s:ACP} the second reduction in parallel to Section \ref{sss:2}
is now straightforward. 

Let $A_1,\ldots,A_r\in \bone$ be trace class operators. 
Then by the very definition of the higher
Fredholm determinant \eqref{eq:I4} we compute analogously to 
\eqref{eq:I9}
\begin{equation}\label{eq:5.1}
    \begin{split}
      \detm&\Bl \mpi l r (I + A_l) \Br \cdot \mpi lr \detm(I+A_l)\ii \\
           &     = \exp\Trb{ \Xtmr\Avec }, \\
           & = \exp\Trb{ \Xmr\Avec } \cdot \exp\Trb{ \Zmrr(\Avec) }.
    \end{split}
\end{equation}
with $\Xtmrr,\Xmrr, Z_{m,r}$ defined in \eqref{eq:4.5}, \eqref{eq:4.6},
\eqref{eq:4.7} resp. Cf. also \eqref{eq:4.14}


After these preparations the natural analogue of Theorem \ref{thm:main1}
is
\begin{theorem}\label{thm:main1mf}
\begin{thmenum}
\item For $\Avec\in\bone$ we have 
\begin{equation}\label{eq:I12mf}
\Tr\bl \Xmr\Avec \br = \Tr\bl \Xtmr\Avec\br.
\end{equation}
\item In the free polynomial algebra $\Qar$ generated by $r$ noncommuting
indeterminates $\avec$ the polynomial
\begin{equation}
   \Zmr\avec = \sum_{\substack{\ga\in\Nrs \\ |\ga|<m} } z_\ga(\avec)
\end{equation}
is a sum of commutators.

\item
For $\ga\in\N^r$ the polynomial $z_\ga(\avec)$ \eqref{eq:4.8}
is a homogeneous polynomial in $\avec$ of degree $|\ga|$.

For $\ga\in\Nrs$ the polynomial $z_\ga(\avec)$ is a sum
of commutators. 

\end{thmenum}
\end{theorem}

\begin{remark}
\simpleenum{1. } Clearly we have the implications 
$\text{(3)} \Rightarrow \text{(2)} \Rightarrow \text{(1)}$ 
of the Theorem.
\simpleenum{2. } \eqref{eq:5.1}, \eqref{eq:I12mf} imply
Theorem \ref{thm:main-mf}. Note that (1), (2) of Theorem
\ref{thm:main-mf} follow immediately from \eqref{eq:4.5}
and \eqref{eq:4.6}.
\simpleenum{3. } The argument in Section \ref{ss:3.4}
shows that (2) and (3) of the 
Theorem are in fact equivalent. Namely, replacing
$a_l$ by $t_l\cdot a_l$ and assuming (2)
we have
\[
   \Zmrr(t_1a_1,\ldots,t_r a_r) 
      = \sum_{\substack{ \ga\in\Nrs \\|\ga|<m}} z_\ga(a)\cdot t^\ga
\]
and thus
\[
    0 = \trAt{\bl \Zmrr(t_1a_1,\ldots,t_ra_r) \br}
      = \sum_{\substack{ \ga\in\Nrs \\|\ga|<m}}
           \trc{\cA}\bl z_\ga(a) \br\cdot t^\ga,
\]
and hence $\trc{\cA}\bl z_\ga(a) \br$ vanishes for 
$\ga\in\Nrs$. The formula also immediately implies that
$z_\ga$ is homogeneous of degree $|\ga|$.

\simpleenum{4. } Summing up to prove the main Theorem
\ref{thm:main-mf} it remains to prove (2) of
Theorem \ref{thm:main1mf}.
\end{remark}

\section{Two proofs of Theorem \ref{thm:main1mf} \textup{(2)}}
\label{s:TPT}

\subsection{Formal power series (analytic) proof} \label{ss:FPSP}

Here we argue along the pattern of Section \ref{s:APPF}, but
using formal power series instead of analytic functions.
Namely, looking at \eqref{eq:4.14} and inserting $x$ for each
$t_l$ the left hand side of \eqref{eq:4.14} is 
divisible by $x^m$ (i.e. $\cO(x^m)$). As $\Xmrr$ contains only
monomials of order $\ge m$ also 
$\Xmrr( x\cdot a_1,\ldots,x\cdot a_r )$
is divisible by $x^m$.
Thus the last summand
$\trAt[x] \bl \Zmrr( x\cdot a_1,\ldots,x\cdot a_r )\br$
on the right of \eqref{eq:4.14}
must be divisible by $x^m$ as well. However, since $\Zmrr$ is
of order $<m$, this summand is a polynomial in $x$ of degree
$<m$ and hence must vanish.

\subsection{Algebraic proof} \label{ss:AP}

Combining (1) and (2) of Proposition \ref{l:log-functional-equation}
gives, using the definition \eqref{sigma-map},
\[
\begin{split}
      \sum_{\ga\in\N^r}  \trA\bl z_\ga(a) \br t^\ga 
      & = \sum_{l=1}^r \trAt[t] \Bl  \log (1+ t_l a_l ) \Br\\
      & = \sum_{j=1}^\infty \frac{(-1)^{j-1}}{j} \sum_{l=1}^r \trA( a_l^j )
      \cdot t^j_l
\end{split}
\]
from which the claim follows by comparing the coefficients
of $t^\ga$.

\bibliographystyle{amsalpha-lmp}
\bibliography{localbib,mlbib}

	
\end{document}